\documentclass{article}

\usepackage{hyperref}

\usepackage{graphicx}
\usepackage{caption}
\usepackage{amssymb,amsmath,amsthm,amscd}
\usepackage{verbatim}
\newcommand{\mycomment}[1]{}

\newcommand\mycup[2]{\overset{#2}{\underset{{#1}}{\cup}}}
\newcommand\mysum[2]{\overset{#2}{\underset{{#1}}{\textstyle\sum}}}

\newtheorem{theorem}{Theorem}
\newtheorem{lemma}{Lemma}

\newtheorem{proposition}{Proposition}
\newtheorem*{remark}{Remark}
\let\olddefinition\remark
\renewcommand{\remark}{\olddefinition\normalfont}



\def\R{{\mathbb R}}
\def\Z{{\mathbb Z}}
\def\P{{\mathbb P}}

\def\cR{{\mathcal R}}

\def\cT{{\mathcal T}}

\newcommand{\sign}{\operatorname{sign}}

\renewcommand{\int}{\operatorname{int}}
\newcommand{\lk}{\operatorname{lk}}
\renewcommand{\d}{\partial}

\newcommand{\id}{\operatorname{id}}

\newcommand\rsetminus{\mathbin{\mathpalette\rsetminusaux\relax}}
\newcommand\rsetminusaux[2]{\mspace{-4mu}
  \raisebox{\rsmraise{#1}\depth}{\rotatebox[origin=c]{-20}{$#1\smallsetminus$}}
 \mspace{-4mu}
}
\newcommand\rsmraise[1]{%
  \ifx#1\displaystyle .8\else
    \ifx#1\textstyle .8\else
      \ifx#1\scriptstyle .6\else
        .45%
      \fi
    \fi
  \fi}

\usepackage{xcolor}
\newcommand*{\missingreference}[1]{\colorbox{yellow}{?#1?}}
\newcommand*{\missingcitation}[1]{\colorbox{green}{?#1?}}
\makeatletter
\def\@setref#1#2#3{%
  \ifx#1\relax
   \protect\G@refundefinedtrue
   \nfss@text{\reset@font\missingreference{#3}}
   \@latex@warning{Reference `#3' on page \thepage \space
             undefined}%
  \else
   \expandafter#2#1\null
  \fi}
\def\@citex[#1]#2{\leavevmode
  \let\@citea\@empty
  \@cite{\@for\@citeb:=#2\do
    {\@citea\def\@citea{,\penalty\@m\ }%
     \edef\@citeb{\expandafter\@firstofone\@citeb\@empty}%
     \if@filesw\immediate\write\@auxout{\string\citation{\@citeb}}\fi
     \@ifundefined{b@\@citeb}{\hbox{\reset@font\missingcitation{#2}}
       \G@refundefinedtrue
       \@latex@warning
         {Citation `\@citeb' on page \thepage \space undefined}}%
       {\@cite@ofmt{\csname b@\@citeb\endcsname}}}}{#1}}
\makeatother

\usepackage[caption=false,labelformat=simple]{subfig}

\usepackage{titlesec}

\mycomment{
\titleformat{\section}
  {\normalfont\scshape}{\thesection}{1em}{}

\titleformat{\subsection}
  {\normalfont\scshape}{\thesubsection}{1em}{}

\titleformat{\subsubsection}
  {\normalfont\scshape}{\thesubsubsection}{1em}{}
}

\begin{document}

\title{\vspace{-2cm}On invariants of link maps in dimension four}

\author{Ash Lightfoot}
\date{}  
\maketitle

\begin{abstract}
We affirmatively address the question of whether the proposed link homotopy invariant $\omega$ of Li is well-defined. It is also shown that if one wishes to adapt the homotopy invariant $\tau$ of Schneiderman-Teichner to a link homotopy invariant of link maps, the result coincides with $\omega$.
\end{abstract}

\section{Introduction}


A link map $S^2\cup S^2\to S^4$ is a map from a union of 2-spheres with pairwise disjoint images, and a link homotopy is a homotopy through link maps. To a link map $f$, Kirk (\cite{Ki1}, \cite{Ki2}) assigned a pair of integer polynomials $\sigma(f)=(\sigma_+(f), \sigma_-(f))$ which is invariant under link homotopy and vanishes if $f$ is link homotopic to a link map that embeds either component. He posed the still-open problem of whether $\sigma(f)=(0,0)$ is sufficient to link nulhomotope $f$. In \cite{Li97},  Li sought to define a link homotopy invariant $\omega(f)=(\omega_+(f), \omega_-(f))$ to detect link maps in the kernel of $\sigma$. When $\sigma_\pm(f)=0$, the mod $2$ integer $\omega_\pm(f)$ obstructs embedding by counting (after a link homotopy) weighted intersections between $f(S^2_\pm)$ and its Whitney disks in the complement of $f(S^2_\mp)$.  While the examples with $\sigma(f)=(0,0)$ but $\omega(f)\neq (0,0)$ in that paper were found to be in error by  Pilz (\cite{Pilz}), the latter did not address another issue.  Namely, the proof that $\omega$ is invariant under link homotopy relies implicitly on the assumption that a pair of link homotopic abelian link maps are link homotopic through \emph{abelian} link maps. The first purpose of this note is to prove this assumption correct and that $\omega$ is a link homotopy invariant.

\begin{theorem}\label{prop:liwelldefined}
If $f$ and $g$ are link homotopic link maps such that $\sigma(f)=(0,0)=\sigma(g)$, then $\omega(f)=\omega(g)$.
\end{theorem}

The invariants $\sigma$ and $\omega$  may then be viewed, respectively, as primary and secondary obstructions to link homotoping to an embedding. For the problem of homotoping a map $S^2\to Y^4$ to an embedding, the homotopy invariants  $\mu$ of Wall (\cite{W}) and $\tau$  of Schneiderman and Teichner (\cite{ST}) form an analogous pair of obstructions. Our second purpose  is to show that if one adapts $\tau$ to the setting of link homotopy in the natural way, one obtains $\omega$. 
%
%
For a link map $f:S^2_+\cup S^2_-\to S^4$, where the signs are used to distinguish each component, write $X_\pm = S^4\rsetminus f(S^2_\mp)$ and let $f_{\pm}$ denote the restricted map $f|S^2_{\pm}:S^2_\pm\to X_{\mp}$.

\begin{theorem}\label{thm:omega-equals-li}
Let $f$ be a link map with $\sigma_+(f)=0$. Then  $f$ is link homotopic to a link map $g$ such that $\tau(g_+)$ is defined, and one has that  $\tau(g_+)=0$ if and only if  $\omega_+(f)=0$.
\end{theorem}

Assume all manifolds are equipped with basepoints and orientations arbitrarily unless otherwise specified.

\section{Proof of Theorem \ref{prop:liwelldefined}}%
\mycomment{%
\section[\texorpdfstring{Link homotopy invariance of $\omega$}%
                        {Link homotopy invariance of omega}]
        {Link homotopy invariance of $\omega$}%
}%
A link map $f$ is said to be \textit{abelian} if $\pi_1(X_{+})\cong \Z$ and $\pi_1(X_{-})\cong \Z$, and an abelian link homotopy is a link homotopy through abelian link maps. We will say $f$ is \textit{good} if it is abelian and each restricted map $f_\pm$ is a self-transverse immersion with vanishing signed  self-intersection number. Theorem \ref{prop:liwelldefined} will be shown using the following lemma.


\begin{lemma}\label{lem:reg-htpic}
If $f$ and $g$ are regularly homotopic good link maps such that $\sigma(f)=(0,0)=\sigma(g)$, then $\omega(f)=\omega(g)$.
\end{lemma}

\begin{proof}[Proof of Theorem \ref{prop:liwelldefined}]
A link map $f$ may be first be perturbed so that it restricts to a self-transverse immersion on each component 2-sphere; local cusp homotopies may then be performed so that these immersions each have vanishing signed self-intersection number. Finger moves of $f(S^2_+)$ in the complement of $f(S^2_-)$, followed by finger moves of $f(S^2_-)$ in the complement of $f(S^2_+)$, may then serve to abelianize $\pi_1(X_+)$ and $\pi_1(X_-)$ (see
\cite[p. 205]{C}; also \cite{Ki1}, \cite{Li97}). \mycomment{By performing local cusp homotopies and finger moves, a link map $f$ is link homotopic to a good link map $f'$ (cf. \cite{Ki2}, \cite{Li97}).}
Denote the resulting good link map by $f'$. In \cite{Li97}, the author gives an algorithm for computing the pair of mod 2 integers  $\omega(f')=(\omega_+(f'), \omega_-(f'))$ and  defines $\omega(f)=\omega(f')$. Suppose $f''$ is another good link homotopy representative of $f$. Then $f'$ and $f''$ are regularly homotopic by \cite[Theorem 2.4]{Ki2},  so $\omega(f'')=\omega(f')$ by Lemma \ref{lem:reg-htpic}. Thus $\omega$ does not depend on the choice of good representative. Now, if $g$ is link homotopic to $f$, and  $g'$ is a good link homotopy representative of $g$, the same argument shows that $\omega(g')=\omega(f')$, so (as defined) $\omega(f)=\omega(g)$.
\end{proof}

%
\mycomment{ REMARK ABOUT DEFINITION OF TAU USING WHITNEY IMMERSION THEOREM:
Define $\tau(f')=\tau(f')$, where $f'$ is self-transverse, zero euler class, immersion approximation. If $g$ is homotopic to $f$, and $g'$ an approx of $g'$, then $f'$ and $g'$ are homotopic and hence regularly homotopic by Whitney (two immersions of $A^k$ in $X^n$, $n\geq 2k+2$ are homotopic iff reg htpic, this is said on first page (pg 281) of Smale "A Classification of Immersions of the Two-sphere.")
}
%
It remains to prove Lemma \ref{lem:reg-htpic}.  We make use of ideas from \cite{K}, to which the reader is referred for more details on finger and Whitney moves along chords. Throughout the rest of this paper, $Y$ will denote a 4-manifold.  Let $k: S^2\to Y$ be an immersion. A \emph{chord} $\gamma$ attached to $k(S^2)$ is a continuous  arc in $Y$ whose endpoints are distinct points of $k(S^2)$ (minus its double points) and whose interior is disjoint from $k(S^2)$. A chord is \emph{simple} if this arc is an embedding. If two simple chords $\gamma$ and $\gamma'$ for $k(S^2)$ are ambient isotopic in $Y$\mycomment{meaning isotopy $Y\times I\to Y$ such that $H_0=\id$, H_1\circ g=h$} by an isotopy fixing $k(S^2)$, then a finger move along either chord yields an ambient isotopic immersion. \mycomment{Note that $H_t$ 1-1 means $F_t(\alpha)\cap F_t(k(S^2)) = F_t(\alpha\cap k(S^2))$, so $\alpha$ is carried to $\beta$ through simple chords automatically.}%
The following result is a ready consequence of transversality and the isotopy extension theorem.
\begin{lemma}\label{lem:ambient}
Let $C$ be a compact subset of a  4-manifold $Y$, and  let $k: S^2\to Y$ be an immersion. Suppose $\alpha$ and $\beta$ are simple chords on $k(S^2)$ with common endpoints $p,q$. If $\alpha$ and $\beta$ are path homotopic in $Y\rsetminus C$ through chords on $k(S^2)$, then $\alpha$ and $\beta$ are ambient isotopic in $Y$ by an isotopy that carries $\alpha$ to $\beta$  and fixes $k(S^2)$ and $C$. \mycomment{Note that $H_t$ 1-1 means $F_t(\alpha)\cap F_t(k(S^2)) = F_t(\alpha\cap k(S^2))$, so $\alpha$ is carried to $\beta$ through simple chords automatically.}
\end{lemma}

We first show that, roughly speaking, finger moves and Whitney moves of a single component of a link map in the complement of the  other component commute.

\begin{lemma}\label{lem:kamada-2-modify}
Let $f$ be a link map such that the restricted maps $f_\pm$ are self-transverse immersions. Suppose that an immersion $g_+$ is obtained from $f_+$ by performing a Whitney move, followed by a finger move, in $S^4\rsetminus f(S^2_-)$. Then, up to ambient isotopy in $S^4$ fixing $f(S^2_-)$, $g_+$ may be obtained from $f_+$ by performing a finger move, followed by a Whitney move, in $S^4\rsetminus f(S^2_-)$.
\end{lemma}
\begin{proof}
By the hypotheses, there is an intermediate link map $f'$ such that $f'_-=f_-$ and a Whitney move performed in a 4-ball $B\subset S^4\rsetminus f(S^2_-)$ changes $f_+$ to $f_+'$, and a finger move of $f_+'$  along a chord $\gamma\subset S^4\rsetminus f(S^2_-)$ changes $f_+'$ to $g_+$.

If $\gamma$ is disjoint from $B$, then the lemma is immediate. Otherwise, we may assume that $\gamma$ intersects $B\rsetminus f'(S^2_+)$ along the interior of $\gamma$ in a collection of $n\geq 0$ properly embedded arcs.    One may then homotop $\gamma$ in a collar of  $\d B\rsetminus f'(S^2_+)$ to a union $\widehat\gamma \cup_{i=1}^n \alpha_i$,  of a simple chord $\widehat\gamma\subset S^4\rsetminus f(S^2_-)$ on $f'(S^2_+)$ that intersects $B$ at precisely one point $p\in \d B$, and $n$ simple loops $\{\alpha_i\}_{i=1}^n$ in $B\rsetminus f'(S^2_+)$ based at $p$. But as inclusion induces a surjection $\pi_1(\d B\rsetminus f'(S^2_+),p)\to \pi_1(B\rsetminus f'(S^2_+),p)$, we may further deduce that $\gamma$ is path homotopic in $S^4\rsetminus f(S^2_+)$ through chords on $f'(S^2_-)$ to a simple chord $\gamma'$ that misses $B$. Thus, by Lemma \ref{lem:ambient} there is an ambient isotopy in $S^4$ from  $\gamma$ to $\gamma'$ that fixes $f'(S^2_+)$ and $f(S^2_-)$.\mycomment{So have $\gamma$ is disjoint from $B$, without moving anything.} \qedhere %
\mycomment{ EXPLANATION for surjectivity: let $\alpha$ be a loop in $B\rsetminus f'(S^2_+)$. It bounds a disk in $B$; make the disk transverse to $f'(S^2_+)$. Then you can shrink $D$ into a loop union a bunch of meridinal disks to $f'$, the latter of which can be pushed into $\d B\rsetminus f'$ along $f'$.}
\end{proof}

We further require that, roughly speaking, a Whitney move of one component of a link map commutes with a finger move of the other. The proof is similar to that of Lemma \ref{lem:kamada-2-modify} but we include it for completeness.

\begin{lemma}\label{lem:kamada-similar}
Suppose that $f$ and $g$ are good link maps such that $g_+$ is obtained from $f_+$ by performing a Whitney move in $S^4\rsetminus f(S^2_-)$ and $g_-$ is obtained from $f_-$ by performing a finger move in $S^4\rsetminus g(S^2_+)$. Then \emph{(}up to ambient isotopy in $S^4$ fixing  $f(S^2_+)$\emph{)} $g_-$ may be obtained from $f_-$ by performing a finger move in $S^4\rsetminus f(S^2_+)$ and \emph{(}up to ambient isotopy in $S^4$  fixing $g(S^2_-)$\emph{)} $g_+$ may be  obtained by performing a Whitney move  of $f_+$ in $S^4\rsetminus g(S^2_-)$.
\end{lemma}
\begin{proof}
Let $B$ be a 4-ball in $S^4\rsetminus f(S^2_-)$ such that a Whitney move performed in $B$ changes $f_+$ to $g_+$, and let $\gamma$ be a simple chord in $S^4\rsetminus g(S^2_+)$ on $f(S^2_-)$ such that a finger move of $f_-$ along $\gamma$ changes $f_-$ to $g_-$.  If $\gamma$ is disjoint from $B$ then the lemma holds without the need for an additional isotopy (note that $f(S^2_+)$ and $g(S^2_+)$ coincide outside $B$)\mycomment{Note that if $\gamma$ is disjoint from $B$, then it is also disjoint from $f(S^2_+) = (f(S^2_+)\cap B) \cup (g(S^2_+)\rsetminus B)$}. Otherwise, we may assume that $\gamma$  intersects $B\rsetminus g(S^2_+)$ along the interior of $\gamma$ in a finite collection of properly embedded arcs.
\mycomment{Otherwise, after an isotopy rel endpoints we may assume that $\gamma$ is disjoint from $f(S^2_+)$ and intersects $B$ in a finite collection of simple arcs.}
Since inclusion induces a surjection $\pi_1(\d B\rsetminus g(S^2_+))\to \pi_1(B\rsetminus g(S^2_+))$, as in the proof of Lemma \ref{lem:kamada-2-modify} one may path homotop $\gamma$ through chords on $g(S^2_+)$ to a simple chord that misses $B$. \mycomment{
\begin{lemma}\label{lem:ambient}
Let $C$ be a compact subset of a  4-manifold $Y$.  Suppose $\alpha$ and $\beta$ are simple chords on $k(S^2)$ with common endpoints $p,q$. If $\alpha$ and $\beta$ are path homotopic in $Y\rsetminus C$ through chords on $k(S^2)$, then $\alpha$ and $\beta$ are ambient isotopic in $Y$ by an isotopy that carries $\alpha$ to $\beta$ through chords and fixes $k(S^2)$ and $C$. \mycomment{Note that $H_t$ 1-1 means $F_t(\alpha)\cap F_t(k(S^2)) = F_t(\alpha\cap k(S^2))$, so $\alpha$ is carried to $\beta$ through simple chords automatically.}
\end{lemma}}%
Now apply Lemma \ref{lem:ambient}.
\end{proof}

We can now prove that if two good link maps are regularly link homotopic then they are connected by an \emph{abelian} link homotopy.

\begin{lemma}\label{prop:abelian-reg-htpy}
If $f$ and $g$ are regularly homotopic good link maps,  then there is a regular homotopy from $f$ to $g$ consisting of a sequence of abelian link homotopies that alternately fix one component.
\end{lemma}

\begin{proof}
As in the proof of \cite[Theorem 2.4]{Ki2}, there is a regular homotopy taking $f$ to $g$ consisting of a sequence of regular homotopies that alternately fix one component.  By Lemmas \ref{lem:kamada-2-modify} and \ref{lem:kamada-similar} this sequence can be chosen to first consist of finger moves (and ambient isotopies) alternately fixing one component, carrying $f$ to a link map $f'$, then a sequence of Whitney moves (and ambient isotopies) alternately fixing one component, carrying $f'$ to $g$.

Now, finger moves and ambient isotopy preserve abelianess, so $f$ and $f'$ are connected by a sequence of abelian, regular link homotopies alternately fixing one component. On the other hand, $f'$ is obtained from $g$ by a sequence of finger moves (and ambient isotopies) alternatively fixing one component, so these link maps are also connected by abelian, regular link homotopies alternately fixing one component. 
\end{proof}

\begin{proof}[Proof of Lemma \ref{lem:reg-htpic}]
By Lemma \ref{prop:abelian-reg-htpy}, the link map $f$ is carried to $g$ by a sequence of abelian, regular link homotopies that alternately fix each component. Lemma \ref{lem:reg-htpic} then follows from the following proposition, which can be deduced from the proof of \cite[Proposition 4.2]{Li97} (and which is expounded upon in \cite[Satz 4.14]{Pilz}).
\end{proof}

\begin{proposition}\label{prop:Li}
Let $h$ be a good link map with $\sigma_+(h)=0$.
\begin{itemize}\itemsep -0.1cm
    \item[(i)]  If a good link map $h'$ is obtained from  $h$ by performing a regular homotopy of $h_+$  in $S^4\rsetminus h(S^2_-)$, then $\omega_+(h')=\omega_+(h).$\hspace{\stretch1}\ensuremath\qedsymbol
    \item[(ii)] \mycomment{$\omega_+$ is invariant under a regular homotopy of $f_-$ in $S^4\rsetminus f(S^2_+)$ to a map $g_-$ \emph{if $\pi_1(S^4\rsetminus g(S^2_-))$ is abelian;}}%
        If a good link map $h'$  is obtained from $h$ by performing a regular homotopy of $h_-$ in $S^4\rsetminus h(S^2_+)$, then $\omega_+(h')=\omega_+(h)$.

\end{itemize}
\end{proposition}

\begin{remark}
Part (i) of this proposition is essentially a special case of the proof in $\cite{ST}$ that the $\tau$-invariant is well-defined, while part (ii) is unique in that the ambient manifold $X_-$, into which $h_+$ maps, is allowed to change.
\end{remark}

\section{Proof of Theorem \ref{thm:omega-equals-li}}

We begin with some preliminary definitions concerning algebraic intersections of immersed surfaces in 4-manifolds. The reader is referred to \cite{FQ} for more details on the subject.

\subsection{Intersection numbers in 4-manifolds}

Suppose $A$ and $B$ are properly immersed, self-transverse 2-spheres or 2-disks in a 4-manifold $Y$. Suppose further that $A$ and $B$ are transverse and that each is equipped with a path (a \emph{whisker}) connecting it to the basepoint of $Y$ . 

For an intersection point $x\in A\cap B$, let $\lambda(A,B)_x \in \pi_1(Y)$ denote the homotopy class of a loop that runs from the basepoint of $Y$ to $A$ along its whisker, then along $A$ to $x$, and back to the basepoint along $B$ and its whisker. Define $\sign_{A,B}(x)$ to be $1$ or $-1$ depending on whether or not, respectively, the orientations of $A$ and $B$ induce the orientation of $Y$ at $x$. The (algebraic) intersection ``number'' $\lambda(A,B)$ between $A$ and $B$  is then defined as the sum in the group ring $\Z[\pi_1(Y)]$ of $\sign(x)\lambda(A,B)_x$ over all such intersection points. The value of $\lambda(A,B)$ is invariant under homotopy rel boundary of $A$ or $B$ (\cite{FQ}), but depends on the choice of basepoint of $Y$ and the choices of whiskers and orientations.

The following two observations will be useful. If $x, y\in A\cap B$, then the product of $\pi_1(Y)$-elements $\lambda(A,B)_x\hskip0.02cm(\lambda(A,B)_y)^{-1}$ is represented by a loop that runs from the basepoint to $A$ along its whisker, along $A$ to $x$, then along $B$ to $y$, and back to the basepoint along $A$ and its whisker.
Secondly, if $D_A\subset A$ is a 2-disk that is equipped with the same whisker  and oriented consistently with $A$, then $\lambda(A,B)_x = \lambda(D_A,B)_x$ and $\sign_{A,B}(x)=\sign_{D_A,B}(x)$ for each $x\in D_A\cap B$.

\subsubsection{Surgering tori to 2-spheres}\label{surger}

Suppose $T$ is an embedded torus {(}or punctured torus, resp.{)} in $Y\setminus \int B$ and suppose there is a circle $\delta_1\subset T$ that is nulhomotopic in $Y$. Choose an immersed 2-disk $D$ in $Y$ that is bound by $\delta_1$ and transverse to $T$ and $B$, and choose a normal vector field $\phi$ to $\delta_1$ on $T$. Let $\delta_1'\subset T$ denote a nearby push-off of $\delta_1$  along $\phi$. Extend $\phi$ over $D$ and let $D'$ denote a pushoff of $D$ along $\phi$, bound by $\delta_1'$ and which we may assume is also transverse to $B$\mycomment{bend $\phi$ near the intersects if necessary}.  If $D$ is oriented and $D'$ has  the orientation induced as the pushoff, then  intersections between $B$ and $D\cup D'$ occur as finitely many nearby pairs of points $\{x_i, x_i'\}_{i=1}^n$, where $x_i\in \int D$ and $x_i'\in \int D'$ are of opposite sign.  Thus, removing from $T$ the interior of the annulus bound by $\delta_1 \cup \delta_1'$ and attaching $D\cup D'$ yields an immersed 2-sphere (or 2-disk with boundary $\d T$, resp.) $S$ in $Y$ such that the intersections between $B$ and $S$ are transverse and occur precisely at the pairs of points $\{x_i, x_i'\}_{i=1}^n$. Furthermore, the algebraic intersections between $S$ and $B$ may be calculated using the following lemma. Let $[\alpha]$   denote the class in $\pi_1(Y)$ of a based loop $\alpha$ in $Y$, let $\overline{\gamma}$ denote the reverse of a path $\gamma$, and let $\ast$ denote composition of paths.

\begin{lemma}\label{lem:lambda-surgery}
 Let $\delta_2$ be an oriented, simple circle on $T$ that intersects each of $\delta_1$ and $\delta_1'$ exactly once, at points $z$ and $z'$ (respectively), and is tangent to $\phi$ at $z$. Let $\iota$ be  a path in $Y$ from its basepoint to $z$. If $S$ and $D$ are oriented consistently and both equipped with the whisker $\iota$, then
\[
    \lambda(S, B) = (1 - [\iota\ast\delta_2\ast\overline{\iota}])\lambda(D, B).
\]
\end{lemma}

\begin{proof}
For each $i$, let $\gamma_i$  be a path on $D$ connecting $z$ to $x_i$ (that does not pass through any double points) and let $\gamma_i'$ be its pushoff along $\phi$, connecting $z'$ to $x_i'$. Let $\beta_i$ be a path on $B$ from $x_i$ to $x_i'$ (that does not pass through any double points), and let $\widehat \delta_2$ be the arc $\delta_2\cap S$, oriented to run from $z'$ to $z$. Then the product $\lambda(S,B)_{x_i}(\lambda(S,B)_{x_i'})^{-1}$ is represented by the loop
\begin{align}\label{eqn:loop}
    \iota\ast \gamma_i \ast \beta_i\ast \overline{\gamma_i'}\ast \widehat\delta_2\ast \overline{\iota}.
\end{align}
Homotoping $S$ (rel boundary) by collapsing $D'$ onto $D$ except near its intersections with $B$\mycomment{because want transverse to $A$}, one sees that the loop \eqref{eqn:loop} is homotopic in $Y$ to the loop  $\iota\ast\delta_2\ast\overline{\iota}$. Thus, equipping $D$ with the whisker $\iota$ and the same orientation as $S$, we have
\[
    \lambda(S,B)_{x_i'} = [\iota \ast\overline{\delta_2}\ast\overline{\iota}] \lambda(S,B)_{x_i} = [\iota \ast \overline{\delta_2}\ast\overline{\iota}] \lambda(D,B)_{x_i}
\]
and $\sign_{S,B}(x_i') = \sign_{D',B}(x_i') = -\sign_{D,B}(x_i)$. Summing over all such pairs of intersections yields
\begin{align*}
    \lambda(S,B) &= \mysum{i}{} \sign_{S,B}(x_i)\lambda(S,B)_{x_i} + \sign_{S,B}(x_i')\lambda(S,B)_{x_i'}\\
    &= \mysum{i}{} (1 - [\iota \ast\overline{\delta_2}\ast\overline{\iota}])\sign_{D,B}(x_i)\lambda(D,B)_{x_i}\\
    &= (1 - [\iota \ast\overline{\delta_2}\ast\overline{\iota}])\lambda(D,B).\qedhere
\end{align*}
\mycomment{ explaining change of basepoint:
If basepoint is $z$. Then $\lambda(S,B)_x\overline{\lambda(S,B)}_x$ is $z$ to $x$ along $D$, to $x'$ along $B$, along $D'$ to $x$, along $\delta_2$ to $z$. Homotopic to $z$ to $x$ along $D$, to $x$ along $B$, along $D$ to $z$, $\delta_2$ to $z$, which is homotopic to $\delta_2$.
If basepoint is $z'$. Then conjugate by path $z'$ to $z\in \delta_2$.
}
\end{proof}

\mycomment{
\begin{remark}
Referring to the notation above, the construction of $S$ required the data $T$, $\delta_1$, $D$ and $\phi$
\end{remark}
{\color{red} Referring to the notation of Lemma \ref{lem:lambda-surgery}, we shall say that $T'$ is the result of surgery on $T$ along $\delta_1$.}
}

\mycomment{
\begin{lemma}\label{lem:lambda-surgery}
Let $B$ be a 2-disk (rel $\d$) or 2-sphere in a 4-manifold $Y$, and let $T$ be an embedded torus (punctured torus) in the complement of $B$.
Suppose $\delta_1\subset T$ is a curve that is nulhomotopic in $Y$ and let $\delta_2\subset T$ be a curve that is dual to $\delta_1$.  
Then surgery on $T$ along $\delta_1$ yields a 2-sphere (or 2-disk) in $Y$ such that
\[
    \lambda(A, B) = (1 - [\delta_2])\lambda(D, B),
\]
where $D$ is a 2-disk in $X$ bound by $\delta_1$ oriented appropriately, and $[\cdot]$ denotes the class in $\pi_1(Y,\ast)$.
\end{lemma}
}

\subsection{Unknotted immersions and link maps}

Two immersions $k_0, k_1:S^2\to \R^4$ are said to be \emph{equivalent} if there are orientation-preserving self-diffeomorphisms $h$ of $S^2$ and $H$ of $\R^4$, respectively, such that $k_1\circ h = H\circ k_0$. Denote the standard embedding $S^2\subset \R^4$ by $u_0^0$. By applying local cusp homotopies, $d$ of positive sign and $e$ of negative sign, to $u_0^0$, one obtains an \emph{unknotted} immersion, denoted $u_d^e:S^2\to \R^4$. Note that $u_d^e$ is unique up to equivalence; we say that an immersion $k:S^2\to \R^4$  (or its image) is unknotted if $k$ is equivalent to $u_d^e$ for some $d, e\geq 0$. See \cite{K} for more details. 

 Identify $S^4 = \R^4 \cup \{\infty\}$.
\begin{lemma}\label{lem:good}
A link map $f$ is link homotopic to a good link map $g$ such that $g(S^2_-)$ is unknotted in $\R^4\subset S^4$.
\end{lemma}

\begin{proof}
As in the proof of Theorem \ref{lem:reg-htpic}, we may assume after a link homotopy that $f$ is a good link map.
By \cite[Lemma 3]{K}, there is a family of disjoint chords attached to $f(S^2_-)$ such that finger moves along them change $f(S^2_-)$ into an unknotted immersion in $\R^4 = S^4 \rsetminus \{\infty\}$. As these chords may be assumed to miss $f(S^2_+)$, we have the required result.
\end{proof}

For an immersed 2-sphere $A$ in a 4-manifold $Y$, let $\omega_2(A)\in \Z_2$ denote the second Stiefel Whitney number of the normal bundle of $A$ in $Y$. The results in \cite{me1} readily generalize to give the following.

\begin{lemma}\label{lem:pi2}
Suppose $k:S^2\to \R^4$ is an unknotted, self-transverse immersion with $d$ double points, and let $Y$ denote the complement in $S^4$ of $k(S^2)$. Then $\pi_2(X_-)$ is a free $\Z[\Z]$-module on $d$ generators, the Hurewisc map $\pi_2(Y)\to H_2(Y)$ surjects  and $\omega_2(A)=0$ for any immersed 2-sphere $A$ in $Y$.
\end{lemma}

\begin{remark}
Indeed, the complement of an open tubular neighborhood of $k(S^2)$ in $S^4$ has a handlebody decomposition consisting of one 0-handle, one 1-handle, and $d$ zero-framed 2-handles attached along unknotted circles in $S^3$ which are nulhomotopic in the boundary of the union of the $0$- and $1$-handle.
\end{remark}

Let $k:S^2\to Y$ be a self-transverse immersion and suppose $p$ is a double point of $k(S^2)$ 
An \emph{accessory circle} for $p$ is an (oriented) simple circle on $k(S^2)$ that passes through exactly one double point, $p$,  and changes sheets there.

\begin{lemma}\label{lem:2-sphere-generators}
Let $f$ be a good link map such that $f(S^2_-)$ is unknotted in $\R^4\subset S^4$. Equip $f(S^2_+)$ with a whisker in $X_-$ and fix an identification of $\pi_1(X_-)$ with $\Z\langle s\rangle$ so as to write $\Z[\pi_1(X_-)] = \Z[s,s^{-1}]$.  Label the double points of $f(S^2_-)$ by $\{p_i\}_{i=1}^d$ and choose an accessory circle  $\alpha_i$ for $p_i$ for each $1\leq i\leq d$. Then $\pi_2(X_-)\cong (\hskip -0.03cm\underset{i=1}{\overset{d}{\oplus}}\Z)[s,s^{-1}]\mycomment{(\Z[\Z])^d}$ and there is a $\Z[s,s^{-1}]$-basis represented by self-transverse, immersed, whiskered 2-spheres $\{A_i\}_{i=1}^d$ in $X_-$ with the following properties. For each $1\leq i\leq d$, there is an integer Laurent polynomial $q_i\in \Z[s,s^{-1}]$ such that
\[
    \lambda(f(S^2_+), A_i) = (1-s)^2q_i(s)
\]
and $q_i(1) = \lk(f(S^2_+),\alpha_i)$. Moreover, if for any $1\leq j\leq d$ the loop $\alpha_j$ bounds a 2-disk in $S^4$ that intersects $f(S^2_+)$ exactly once, then we may choose $A_j$ so that
\[
    \lambda(f(S^2_+), A_j) = (1-s)^2.
\]
\end{lemma}

\begin{proof} 
For $t_0, t_0'\in \R$, $t_0'>t_0$, let $\R^3[t_0]$ denote the hyperplane of $\R^4$ whose fourth coordinate $t$  is $t_0$, and let $\R^3[t_0,t_0'] = \{(x,t)\in \R^4: x\in \R^3, t_0\leq t\leq t_0'\}$.  Figure \ref{fig:model} gives a ``moving picture'' description of an immersed 2-disk $U$ (appearing as an arc in each slice $\R^3[t_0]$, $t_0\in [-1,1]$) in a 4-ball $N\subset \R^3[-1,1]$, with a single self-transverse double point $p\subset \R^3[0]$. In this figure we have labeled a loop $\alpha\subset\R^3[0]$ on $U$ that changes sheets at $p$ and bounds a 2-disk $D$. For each $1\leq i\leq d$, let $\widehat U_i$ be a 2-disk on $S^2_-$ that contains the two preimages of the double point $p_i$, and no other double point preimages.  There is a diffeomorphism $\Gamma_i$ of $N$ onto a 4-ball neighborhood of $p_i$ in $S^4$ that takes $U$ to $f(\widehat U_i)$, $\alpha$ to $\alpha_i$ and $p$ to $p_i$. Choose a 4-ball neighborhood $N^+\subset N$ of $p$ so that (the smaller 4-ball) $\Gamma_i(N^+)$  is disjoint from $f(S^2_+)$. 
\newlength{\myheight}%
\setlength{\myheight}{0.25\textwidth}%
\begin{figure}[h]
\centering
    \includegraphics[height=\myheight]{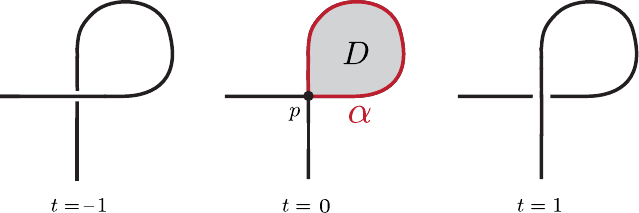}
    \caption{}
    \label{fig:model}
\end{figure}%
There is a torus $T$ in $N^+\rsetminus U$ that intersects $D$ exactly once; see Figure \ref{fig:model-with-torus}. The torus appears as a cylinder in each of $\R^3[-1]$ and $\R^3[1]$, and appears as a pair of circles in $\R^3[t_0]$ for $t_0\in (-1,1)$.
\begin{figure}[h]
\centering
    \includegraphics[height=1.2\myheight]{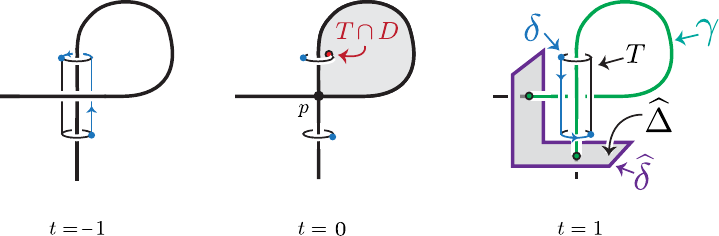}
    \caption{}
    \label{fig:model-with-torus}
\end{figure}
By Alexander duality the linking pairing
\[
    H_2(X_-)\times H_1(f(S^2_-))\to \Z
\]
defined by $(R,\upsilon)\mapsto R\cdot\Upsilon$, where $\upsilon=\d \Upsilon\subset S^4$, is nondegenerate. Thus, as the loops $\{\alpha_i\}_i$ represent a basis for $H_1(f(S^2_-))\cong \Z^{d}$, we have that $H_2(X_-) \cong \Z^{d}$ and (after orienting) the so-called \textit{linking tori}  $\{T_i\}_i$, defined by $T_i=\Gamma_i(T)$, represent a basis.  
We proceed to apply the construction of \eqref{surger} (twice, successively) to turn these tori into 2-spheres.

In Figure \ref{fig:model-with-torus} we have illustrated an oriented circle $\delta$ on $T$ which intersects $\R^3[-1]$ and $\R^3[1]$ each in an arc, and appears as a pair of points in $\R^3[t_0]$ for $t_0\in (-1,1)$. 
Notice that $\delta$ is isotopic in $N^+\rsetminus U$ to the circle $\widehat \delta\subset \R^3[1]$ that is also illustrated in Figure \ref{fig:model-with-torus}. By attaching the trace of such an isotopy to the 2-disk $\widehat \Delta\subset \R^3[1]$ illustrated, bound by $\widehat \delta$, one may obtain an embedded 2-disk $\Delta\subset N^+$ that is bound by $\delta$ and intersects $U$ precisely where $\widehat \Delta$ does. 
These intersection points are the endpoints of an arc $\gamma\subset \R^3[1]$ on $U$, shown in Figure \ref{fig:model-with-torus}. Let $\gamma_i = \Gamma_i(\gamma)\subset f(S^2_-)$.  In  Figure \ref{fig:tubular-nbd} we have illustrated in $\R^3[1]$  the restriction of a tubular neighborhood of $U$  to $\gamma$. Identifying this tubular neighborhood with $\gamma\times D^2$, we may assume the embedding $\Gamma_i$ carries $\gamma\times D^2$  onto the restriction over $\gamma_i$ of a tubular neighborhood of $f(S^2_-)$ that  is disjoint from $f(S^2_+)$.
\begin{figure}%
    \centering
    \subfloat[]{\label{fig:tubular-nbd}\includegraphics[height=1.3\myheight]{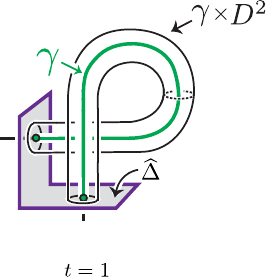}}%
    \hskip 1cm
    \subfloat[]{\label{fig:tubular-nbd-with-beta}\includegraphics[height=1.3\myheight]{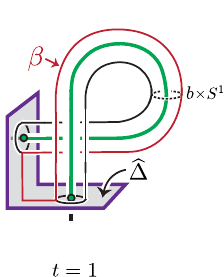}}%
    \caption{}%
    \label{}%
\end{figure}
Let $\Theta$ be the punctured torus in $N\rsetminus  U$ given by
\[
    \Theta = \Delta \rsetminus (\d \gamma\times \int D^2) \mycup{\d\gamma\times S^1}{} (\gamma\times S^1),
\]
which has boundary $\delta$. Note that $\Gamma_i(\Theta)$ is disjoint from $f(S^2_+)$. Form a loop $\beta$ on $\R^3[1]\cap  \Theta$ by connecting the endpoints of $\gamma\times \{1\}$ by an arc on $\widehat \Delta \rsetminus (\d \gamma\times \int D^2)$.  Since $\pi_1(X_-)\cong \Z$ and a loop of the form $\Gamma_i(\{b\}\times S^1)$ ($b\in \int \gamma$) is meridinal to $f(S^2_-)$, by replacing $\gamma\times \{1\}\subset \beta$ by its band sum with oriented copies of $\{b\} \times S^1$ if necessary (see Figure \ref{fig:tubular-nbd-with-beta}) we may assume that $\beta_i=\Gamma_i(\beta)$ is a simple circle bounding an immersed, self-transverse 2-disk $D_i$ in $X_-$ that is transverse to $f(S^2_+)$.

Now, as $f(S^2_+)$ misses $\Gamma_i(N^+)$, the loop $\beta_i$ is freely homotopic in $S^4\rsetminus f(S^2_+)$ to $\alpha_i$.  Consequently,
\begin{equation}\label{eq:f-D-i}
    |f(S^2_+)\cdot D_i| = |\lk(f(S^2_+), \alpha_i)|
\end{equation}
as non-negative integers. Let $\widehat A_i$ be the immersed, self-transverse 2-disk in $X_-$ obtained by performing the construction of \eqref{surger} with the (embedded) punctured torus $\Gamma_i(\Theta)\subset X_-\rsetminus f(S^2_+)$, $\beta_i\subset \Gamma_i(\Theta)$, $D_i$ and some choice of normal vector field to $\beta_i$ on $\Gamma_i(\Theta)$.
Then, since a loop of the form $\Gamma_i(\{b\}\times S^1)$ ($b\in \int \gamma$) is dual to $\beta_i$ on $\Gamma_i(\Theta)$ and hence represents a generator ($s$ or $s^{-1}$) of $ \pi_1(X_-)$, by Equation \eqref{eq:f-D-i} and Lemma \ref{lem:lambda-surgery} we have (after orienting $\widehat A_i$ and connecting it to the basepoint of $X_-$)
\begin{equation}\label{eq:f-hat-A-i}
    \lambda(f(S^2_+), \widehat A_i) =(1-s)\widehat q_i(s)
\end{equation}
for some integer Laurent polynomial $\widehat q_i\in \Z[s,s^{-1}]$ such that $\widehat q_i(1)=\lk(f(S^2_+), \alpha_i)$. Moreover, if $| f(S^2_+) \cap D_j|=1$ for some $j$ then (since we are free to choose the orientation and whisker of $\widehat A_j$)  we may take $\widehat q_j=1$. %

Now, for each $i$, $\widehat A_i$ is bound by the circle $\Gamma_i(\delta)$ on the (embedded) linking torus $T_i\subset X_-\rsetminus f(S^2_+)$. Perform the construction of \eqref{surger} with $T_i$, $\Gamma_i(\delta)$, $\widehat A_i$ and some choice of normal vector field to $\Gamma_i(\delta)$ on $T_i$.  Then, since $\delta$  has a dual curve on $T$ that is meridinal to $U$,  by  Equation \eqref{eq:f-hat-A-i} and Lemma \ref{lem:lambda-surgery} we have (after orienting $A_i$ and connecting it to the basepoint of $X_-$)
\begin{equation*}
    \lambda(f(S^2_+), A_i) =(1-s)^2 q_i(s)
\end{equation*}
for some $q_i\in \Z[s,s^{-1}]$ such that $q_i(1)=\lk(f(S^2_+), \alpha_i)$. As above, if  $| f(S^2_+) \cap D_j|=1$ for some $j$ then  we may take $q_j=1$.

By construction, $A_i$ is homologous to $T_i$  for each $i$, so by Lemma \ref{lem:pi2} the immersed 2-spheres $\{A_i\}_{i=1}^d$ represent a $\Z[s,s^{-1}]$-basis  for $\pi_2(X_-)$.\qedhere
\mycomment{, It is a well known result that the complement of a self-transverse, immersed 2-sphere in $S^4$ (or, equivalently, a self-transverse, properly immersed 2-disk in $D^4$) has a cellular decomposition consisting of one $0$-cell, $1$-cell and $d$ $2$-cells. From this the hypothesis that $\pi_1(X_-)\cong \Z$ implies that $\pi_2(X_-)=(\Z[\Z])^{d}$. See, for example, \cite{me1}. To see that the collection $\{A_i\}_i$ are generators, one notes the exact sequence
\[
    \pi_2(X_-)\xrightarrow{\text{Hurewicz}} H_2(X_-)\to H_2(\pi_1(X_-))=0.\qedhere
\]}

\subsection[\texorpdfstring{The invariant $\tau$ applied to link maps}%
                        {The invariant tau applied to link maps}]{The invariant $\tau$ applied to link maps}

In \cite{ST}, the authors define a homotopy invariant $\tau$ which takes as input a map $k:S^2\to Y^4$ with vanishing Wall self-intersection $\mu(k)$  and gives output in a  quotient $\Pi(Y, k)$ of the group ring $\Z[\pi_1(Y)\times \pi_1(Y)]$ modulo certain relations. The relations are additively generated by the equations
\begin{align}
    (a,b)&=-(b,a)\tag{$\cR_1$}\\
    (a,b)&=-(a^{-1}, ba^{-1})\tag{$\cR_2$}\\
    (a,1)&=(a,a)\tag{$\cR_3$}\\
    (a,\lambda(k(S^2),A))&=(a,\omega_2(A)\cdot 1)\label{eq:omega-2-relation}\tag{$\cR_4$}
\end{align}
where $a,b\in \pi_1(Y)$, $A$ represents an immersed $S^2$ or $\R \P^2$ in $Y$ (in the latter case, the group element $a$ is the image of the nontrivial element in $\pi_1(\R\P^2)$).

Let $f$ be a good link map with $\sigma_+(f)=0$ (from which it follows that $\mu(f_+)=0$). For an integer $k$, let $\overline{k}$ denote its image in $\Z_2$. Letting $\rho$ denote the mod $2$ Hurewicz  map $\pi_1(X_-)\to H_1(X_-;\Z_2)=\Z_2$, define a ring homomorphism $\varphi_f: \Pi(X_-, f_+) \to \Z_2\langle t: t^2=1\rangle$ by
\[
    (a,b)\mapsto t^{\overline{\rho(a)+\rho(a)\rho(b)+\rho(b)}},
\]
and extending linearly mod $2$. We now prove a stronger form of Theorem \ref{thm:omega-equals-li}.

\begin{lemma}
Let $f$ be a link map with $\sigma_+(f)=0$. After a certain link homotopy of $f$ we have that  $\varphi_f$ is an isomorphism and takes $\tau(f_+)$ to $(1+t)\hskip0.03cm\omega_+(f)$.
\end{lemma}

\mycomment{
\begin{lemma}[{{\cite[Proposition 4.3]{me1}}}]\label{lem:newreln} 
If $Y$ is the complement in $S^4$ of an unknotted, self-transverse 2-sphere, then the relations \eqref{eq:omega-2-relation} may be replaced by the relations additively generated by the equations
\begin{align}
    (a,\lambda(k,A)) = (0,0)\tag{$\cR_4'$}
\end{align}
for $a\in\pi_1(Y)$ and $A\in \pi_2(Y)$.\qedhere
\end{lemma}
}

\begin{proof}
By Lemma \ref{lem:good} we may assume $f$ is a good link map (and so $\mu(f_+)=\sigma_+(f)=0$) such that $f(S^2_-)$ is unknotted. We may perform a finger move of $f(S^2_-)$ along an chord attached in the complement of and meridinal to $f(S^2_-)$ so that a slice in $\R^3[t_0]$ (for some $t_0$) of the result  is illustrated in Figure \ref{fig:finger-moved-dbl-pt}. This produces a pair of oppositely-signed double points $\{p^+, p^-\}$ on $f(S^2_-)$ such that (in particular) $p^+$ has an accessory circle bounding an obvious embedded 2-disk in $\R^3[t_0]$ that intersects $f(S^2_+)$ exactly once. Note that $f(S^2_-)$ is still unknotted (by \cite[Lemma 1]{K}) and, in particular, its complement in $S^4$ still has abelian fundamental group.

Now, fixing $f_-$ and $X_-$, by Lemma \ref{lem:2-sphere-generators} we may thus identify $\pi_1(X_-)$ with $\Z\langle s\rangle$ and $\pi_2(X_-)$ with $(\hskip -0.03cm\underset{i=1}{\overset{d}{\oplus}}\Z)[s,s^{-1}]$, for some $d\geq 0$, such that  there is an immersed,   whiskered 2-sphere $A_0$ in $X_-$ with the property that $\lambda(f(S^2_+), A_0)=(1-s)^2$. Moreover,  for any whiskered, immersed 2-sphere $A$ in $X_-$ we have $\lambda(f(S^2_+), A)=(1-s)^2q_A(s)$ for some integer Laurent polynomial $q_A\in \Z[s,s^{-1}]$.
\begin{figure}[h]
\centering
    \includegraphics[width=.30\textwidth]{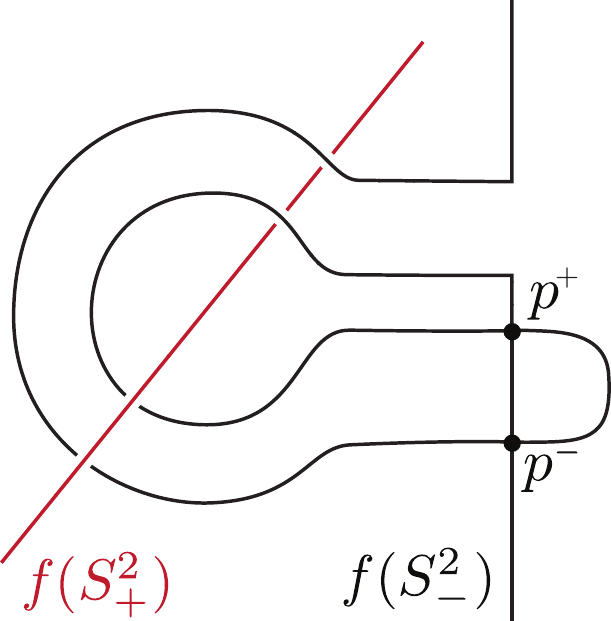}
    \caption{}
    \label{fig:finger-moved-dbl-pt}
\end{figure}
%
%
\mycomment{
Now, from the proof of Lemma \ref{lem:2-sphere-generators} we see that  $H_2(X_-)$ is generated by linking tori which have vanishing homological self-intersection. It follows from the Wu formula that the second Stiefel Whitney class of the tangent bundle $TX_-$ vanishes. Thus, as the tangent bundle of $S^2$ has vanishing Stiefel Whitney numbers, the Cartan formula implies that for any immersion $h:S^2\to X_-$ one has $\omega_2(A) = h^\ast\omega_2(TX_-) = 0$. \mycomment{ Or cite (c.f. \cite{me1} Proposition 4.3) }
}

Therefore, if we identify $\Z[\pi_1(X_-)\times \pi_1(X_-)]$ with $\Z[s^{\pm 1},t^{\pm 1}]$ via $(s^n,s^m)=s^nt^m$ (for $n, m\in \Z$), by Lemma \ref{lem:pi2} the  ring $\Pi(X_-,f_+)$ is the quotient of  the group ring $\Z[s^{\pm 1},t^{\pm 1}]$ modulo the relations  generated additively by the equations:
\begin{align}
    s^{n}t^{n} - s^{n} &= 0\tag{$\cT_1$}\label{reln1}\\
    s^{n}t^{m}+s^{-n}t^{m-n} &= 0\tag{$\cT_2$}\label{reln2}\\
    s^{n}t^{m}+s^{m}t^{n}&=0\tag{$\cT_3$}\label{reln3}\\
    s^{n}t^m(1-t)^2&=0\tag{$\cT_4$}\label{reln4}
\end{align}
where $n, m\in \Z$. Note that in reformulating Relation \eqref{eq:omega-2-relation} to obtain  Relation \eqref{reln4} we have used the action of $\pi_1(X_-)=\Z\langle s\rangle$ on $\pi_2(X_-)$.

Let $\equiv$ denote equivalence in $\Pi(X_-,f_+)$. Clearly $\varphi_f$ is surjective; to show injectivity we first show that for any integers $n, m$, one has
\begin{align}\label{reln-big}
    s^nt^m \equiv t^{\overline{n+nm+m}}.
\end{align}
By Relations \eqref{reln2} and \eqref{reln3} we have $2t^n\equiv 0$ and hence $2s^n\equiv -2t^n\equiv 0$ for each $n\in \Z$.  Then Relation \eqref{reln4} implies that $t^{m+2}\equiv t^m$ for any integer $m$, and it follows by an induction argument that
\begin{align}\label{tm}
    t^m\equiv t^{\overline{m}}.
\end{align}
Now, $s\equiv -t\equiv t$ and $st\equiv t$ by Relations \eqref{reln1}-\eqref{reln3}. Combining these equivalences with the consequence of Relation \eqref{reln4} that $st^{m+2}\equiv 2st^{m+1}-st^m$, an induction gives
\begin{align}\label{stm}
    st^m\equiv t
\end{align}
for any integer $m$.

Finally, fix $n_0\in \Z$.  By Relation \eqref{reln3} and Equivalences \eqref{tm} and \eqref{stm},  we have $s^{n_0}\equiv -t^{n_0} \equiv t^{\overline{n_0}}$ and $s^{n_0}t \equiv -st^{n_0} \equiv -t \equiv t$. 
Suppose now that for some $k\geq 1$ Equivalence \eqref{reln-big}  holds for $n=n_0$ and any $m\in\{0,1,\ldots, k\}$. Then Relation \eqref{reln4} implies that
\begin{align*}
s^{n_0}t^{k+1}&\equiv 2s^{n_0}t^{k}-s^{n_0}t^{k-1}\\
&\equiv 2t^{\overline{{n_0}+{n_0}k+k}}-t^{\overline{{n_0}+{n_0}(k-1)+(k-1)}}\\
&\equiv t^{\overline{{n_0}+{n_0}(k+1)+(k+1)}}.
\end{align*}
On the other hand, suppose that for some $k\leq 0$ Equivalence \eqref{reln-big} holds for $n=n_0$ and any $m\in\{k, k+1, \ldots, 0, 1\}$; then
\begin{align*}
s^{n_0}t^{k-1}&\equiv 2s^{n_0}t^{k}-s^{n_0}t^{k+1}\\
&\equiv 2t^{\overline{{n_0}+{n_0}k+k}}-t^{\overline{{n_0}+{n_0}(k+1)+(k+1)}}\\
&\equiv t^{\overline{{n_0}+{n_0}(k-1)+(k-1)}}.
\end{align*}
Thus, by induction Equivalence \eqref{reln-big} holds for $n=n_0$ and any integer $m$. But $n_0\in \Z$ was arbitrary, so the equivalence holds for all integers $n, m$.  As $2\equiv 0\equiv 2t$, we deduce that $\Pi(X_-, f_+)$ is the group ring $\Z_2\langle t: t^2=1\rangle$ and $\varphi_f$ is injective.

Turning to the second part of the lemma, we refer the reader to \cite{Li97} and \cite{ST} for detailed descriptions of the $\omega$ and $\tau$ invariants, respectively, and to \cite{FQ} for background on framed Whitney disks. We make only a few summarizing remarks.

Since $\sigma_+(f)=0$, $\pi_1(X_-)=\Z\langle s\rangle$, and  $f_+$ is self-transverse with vanishing signed sum of its double points, the double points of $f(S_+^2)$  may be decomposed into \emph{canceling} pairs $\{p_i^+,p_i^-\}_{i=1}^k$ in the following sense.  For each $1\leq i\leq k$, one has $\sign(p_i^+)=-\sign(p_i^-)$ and the preimages of $p_i^\pm$ in $S^2_+$ may be labeled $\{x_i^\pm, y_i^\pm\}$ so that if $\gamma_i$ is an arc on $S^2_+$ connecting  $x_i^+$ to $x_i^-$ (and missing all other double point  preimages) and $\gamma_i'$ is an arc on $S^2_+$ connecting  $y_i^+$ to $y_i^-$ (and missing $\gamma_i$ and all other double point  preimages),  then the loop $f(\gamma_i)\cup f(\gamma_i')\subset f(S^2_+)$ is nulhomotopic in $X_-$. The arcs $\{\gamma_i,\gamma_i'\}_{i=1}^k$ may be chosen  so that the resulting \emph{Whitney circles} $\{f(\gamma_i\cup \gamma_i')\}_{i=1}^k$ are mutually disjoint, simple circles in $X_-$ such that each bounds an immersed, framed Whitney disk $W_i$ in $X_-$ whose interior is transverse to $f(S^2_+)$. Let $\alpha_i^\pm$ be an arc on $S^2_+$ connecting $x_i^\pm$ to $y_i^\pm$, and let $n_i^\pm$ denote the integer $\lk(f(S^2_-), f(\alpha_i^\pm))$; then  $n_i^+=-n_i^-$. (In \cite{Li97} the non-negative integer $|n_i^+|$ is called the \emph{$n$-multiplicity} for the pair $\{p_i^+, p_i^-\}$, and in \cite{ST} the $\pi_1(X_-)$-element $s^{n_i^+}$ is called the \emph{primary group element} for $W_i$.) Note that $\rho(s^{n_i^+})$ is the mod $2$ image of $n_i^+$.

\mycomment{We will  refer to a neighborhood in $f(S^2_+)$ of $f(\gamma_i)$ (resp. $\gamma_i''$) as the \emph{positive} (resp. \emph{negative}) \emph{sheet} of $g(S^2_+)$ near $W_i$.}%
\mycomment{ In my translation, since $n_i^+$ goes from neg to pos sheet in \cite{ST}, so $\gamma_i$ is neg arc and $\gamma_i'$ is pos arc}
Let $i\in \{1,2,\ldots, k\}$ and suppose $x\in f(S^2_+) \cap \int W_i$. A loop that first goes along $f(S^2_+)$ from its basepoint to $x$, then along $W_i$ to $f(\gamma_i')\subset \d W_i$, then back along $f(S^2_+)$ to the  basepoint of $f(S^2_+)$, determines a $\pi_1(X_-)$-element $s^{m_x}$ (called the \emph{secondary group element} associated to $x$ in \cite{ST};   the non-negative integer $|m_x|$ is called the $m$-\emph{multiplicity} of $x$ in \cite{Li97}). %
Associate to $x$ a sign by  orienting $W_i$ using the following convention: orient $\d W_i$ from $p_i^-$ to $p_i^+$ along the $f(\gamma_i')$, then back to $p_i^-$ along $f(\gamma_i)$; the positive tangent to $\d W_i$ together with an outward-pointing second vector then orient $W_i$.
\mycomment{
Since  the positive and negative sheets meet transversely at $p_i^\pm$, there are a pair of smooth vector fields $v_1, v_2$ on $\d W_i$ such that $v_1$ is tangent to $f(S^2_+)$ along $\alpha_i$ and normal to $f(S^2_+)$ along $\beta_i$, while $v_2$ is normal to $f(S^2_+)$ along $\alpha_i$ and tangent to $f(S^2_+)$ along $\beta_i$. Such a pair defines a normal framing of $W_i$ on the boundary. We say that $\{v_1,v_2\}$ is a \emph{correct framing} of $W_i$, and that $W_i$ is \emph{framed}, if the pair extends to a normal framing of $W_i$.
}
Let
\[
    J^i_x= \overline{n_i^+ + n_i^+m_x + m_x} \in \Z_2
\]
and
\[
I^i_x = \sign(x) s^{n_i^+}t^{m_x} \in \Z[s^{\pm 1},t^{\pm 1}].
\]
Then Li's $\Z_2$-valued $\omega_+$-invariant applied to $f$  is defined by
\[
    \omega_+(f) = \mysum{i=1}{k}\; \mysum{x\, \in\,  f(S^2_+)\,\cap \,\int W_i}{} J^i_x \mod 2;
\]
while, in this special case, the Schneiderman-Teichner invariant $\tau$ applied to $f_+$ is given by the $\Z[s^{\pm 1},t^{\pm 1}]$-sum
\[
    \tau(f_+) =\mysum{i=1}{k}\; \mysum{x\, \in\,  f(S^2_+)\,\cap \,\int W_i}{} I^i_x
\]
evaluated in the quotient $\Pi(X_-, f_+)$.

Now, $\varphi_f(I^i_x)=  t^{J^i_x}$, and consequently
\begin{align*}
   \varphi_f(\tau(f_+)) &= \mysum{i=1}{k}\; \mysum{x\, \in\,  f(S^2_+)\,\cap \,\int W_i}{} {t^{J^i_x}} \;.\mycomment{12-22-15:\\
    &= \overline{\scalebox{1.3}{\#}\{(i,x)\hspace*{-0.1cm}: J^i_x=0\}}\; +\; \overline{\scalebox{1.3}{\#}\{(i,x)\hspace*{-0.1cm}: J^i_x=1\}}\cdot t,}
\end{align*}
\mycomment{where $1\leq i\leq k$ and $x\in \int W_i\cap f(S^2_+)$.}  But $\varphi(\tau(f_+))\in\{0,1,t,1+t\}$ must map forward to $0$ under the homomorphism $\Pi(X_-,f_+)\to \Pi(S^4,f_+)=\Z_2$ induced by the inclusion $X_-\subset S^4$ and given by sending $s$, $t \mapsto 1$. Thus \mycomment{12-22-15: 
\[
\scalebox{1.3}{\#}\{(i,x)\hspace*{-0.1cm}: J^i_x=0\} = \scalebox{1.3}{\#}\{(i,x)\hspace*{-0.1cm}: J^i_x=1\} \mod 2
\]
and so}
\begin{align*}
    \varphi_f(\tau(f_+)) &\equiv \mycomment{12-22-15: \overline{\scalebox{1.3}{\#}\{(i,x)\hspace*{-0.1cm}: J^i_x=1\}}\cdot(1+t)\\
    &= }\mysum{i=1}{k}\; \mysum{x\, \in\,  f(S^2_+)\,\cap \,\int W_i}{} J_i^x\cdot (1+t) \mod 2\\
    &= (1+t)\hskip0.03cm\omega_+(f).\qedhere
\end{align*}
\end{proof}
\mycomment{ Pre 12-29:
Now, the Equivalence \eqref{reln-big} in $\Pi(X_-, f_+)$ implies that  $I^i_x \equiv t^{J^i_x}$, and consequently
\begin{align*}
    \tau(f_+) &\equiv \mysum{i=1}{k}\; \mysum{x\, \in\,  f(S^2_+)\,\cap \,\int W_i}{} {t^{J^i_x}} \;.\mycomment{12-22-15:\\
    &= \overline{\scalebox{1.3}{\#}\{(i,x)\hspace*{-0.1cm}: J^i_x=0\}}\; +\; \overline{\scalebox{1.3}{\#}\{(i,x)\hspace*{-0.1cm}: J^i_x=1\}}\cdot t,}
\end{align*}
\mycomment{where $1\leq i\leq k$ and $x\in \int W_i\cap f(S^2_+)$.}  But $\tau(f_+)$ must map forward to $0$ under the homomorphism $\Pi(X_-,f_+)\to \Pi(S^4,f_+)=\Z_2$ induced by the inclusion $X_-\subset S^4$ and given by sending $s$, $t \mapsto 1$. Thus \mycomment{12-22-15: 
\[
\scalebox{1.3}{\#}\{(i,x)\hspace*{-0.1cm}: J^i_x=0\} = \scalebox{1.3}{\#}\{(i,x)\hspace*{-0.1cm}: J^i_x=1\} \mod 2
\]
and so}
\begin{align*}
    \tau(f_+) &\equiv \mycomment{12-22-15: \overline{\scalebox{1.3}{\#}\{(i,x)\hspace*{-0.1cm}: J^i_x=1\}}\cdot(1+t)\\
    &= }\mysum{i=1}{k}\; \mysum{x\, \in\,  f(S^2_+)\,\cap \,\int W_i}{} J_i^x\cdot (1+t) \mod 2\\
    &= \omega_+(f)(1+t).\qedhere
\end{align*}
}%
\end{proof}



\begin{thebibliography}{test}

\bibitem{C} Casson  A.
\emph{Three lectures on new infinite constructions in 4-dimensional manifolds.}
 A la  Recherche de la Topologie Perdue, Progress in Mathematics 62, Birkh\"{a}user, Basel (1986)

\bibitem{FQ} Freedman M. H., Quinn F. \emph{The topology of 4-manifolds.} Princeton Math.
Series 39, Princeton, NJ (1990)

\mycomment{\bibitem{H}  Hirsch  M.
\emph{Differential topology.}
Graduate Texts in Mathematics No. 33, Springer-Verlag (1976)
}

\bibitem{K} Kamada S.
\emph{Vanishing of a certain kind of Vassiliev invariants of 2-knots.}
Proc. AMS
127 (11) (1999) 3421-3426

\bibitem{Ki1}  Kirk P. \emph{Link maps in the four sphere.} Proc. 1987 Siegen Topology Conf., SLNM 1350,
Springer, Berlin (1988)

\bibitem{Ki2} Kirk P. \emph{Link homotopy with one codimension two component.} Trans. Amer. Math.
Sot., 319 (1990) 663-688.

\bibitem{Li97}  Li G. S. \emph{An invariant of link homotopy in dimension four.} Topology 36 (1997) 881-897

\bibitem{me1} Lightfoot A.  \emph{The Schneiderman-Teichner invariant applied to immersions arising from link maps in $S^4$.} arXiv:1312.1936 [math.GT]

\mycomment{
\bibitem{me-welldefined} Lightfoot A. \emph{A note on invariants of link maps in dimension four.} Preprint

\bibitem{MR} Massey W. S., Rolfsen D. \emph{Homotopy classification of higher-dimensional links.} Indiana Univ. Math. J. 34 (1985) 375-391




\bibitem{MR} Martin J., Rolfsen D. \emph{Homotopic arcs are isotopic} Proc Amer. Math. Soc. 34
(1986) 177-184

\bibitem{M}  Milnor  J.
\emph{Topology from the differential viewpoint.}
Princeton University Press (1997)
}

\bibitem{Pilz}  Pilz A. {\emph{Verschlingungshomotopie von 2-Sph\"{a}ren im 4-dimensionalen Raum.} Diploma thesis,  University of Seigen (1997)}

\mycomment{
\bibitem{RS} Rourke, C., Sanderson B. \emph{Introduction to piecewise-linear topology.} Springer Science & Business Media ( 2012)
}

\bibitem{ST} Schneiderman R.,  Teichner P. \emph{Higher order intersection numbers of 2-spheres in
4-manifolds.} Algebraic and Geometric Topology, 1 (2001) 1–29

\mycomment{
\bibitem{S1} Schneiderman R. \emph{Algebraic linking numbers of knots in 3–manifold.} Algebraic and Geometric Topology, 3 (2003) 921-968
}

\bibitem{W}  Wall C. T. C. \emph{Surgery on Compact Manifolds.} Academic Press, New York (1970)

\end{thebibliography}
\end{document}